\documentclass[11pt]{amsart}

\usepackage[utf8]{inputenc}
\usepackage{lmodern}
\usepackage[T1]{fontenc}
\usepackage[english]{babel}

\usepackage{latexsym}
\usepackage{amsmath,amssymb,amsthm}
\usepackage{mathrsfs}

\usepackage{enumerate}

\usepackage{rotating}

\usepackage{color}
\usepackage{hyperref}
\hypersetup{
	colorlinks,
	citecolor=black,
	filecolor=black,
	linkcolor=black,
	urlcolor=black
}

\newtheorem*{maintheorem}{Main Theorem}
\newtheorem{theoremAlph}{Theorem}
\newtheorem{theorem}{Theorem}[section]
\newtheorem{lemma}[theorem]{Lemma}	
\newtheorem{proposition}[theorem]{Proposition}

\theoremstyle{definition}
\newtheorem{definition}[theorem]{Definition} 
\newtheorem{remark}[theorem]{Remark}

\theoremstyle{definition} 
\newtheorem*{ack}{Acknowledgements}

\numberwithin{equation}{section}

\newcommand{\C}{\mathbb{C}} % komplexe
\newcommand{\R}{\mathbb{R}} % reelle
\newcommand{\Q}{\mathbb{Q}} % rationale
\newcommand{\Z}{\mathbb{Z}} % ganze
\newcommand{\Rpsc}{\mathcal{R}^+}
\newcommand{\Mpsc}{\mathcal{M}^+}
\newcommand{\SO}{\textup{SO}}
\newcommand{\Spin}{\textup{Spin}}
\newcommand{\Pin}{\textup{Pin}}

\makeatletter
\DeclareRobustCommand*\uell{\mathpalette\@uell\relax}
\newcommand*\@uell[2]{
	\setbox0=\hbox{$#1\ell$}
	\setbox1=\hbox{\rotatebox{10}{$#1\ell$}}
	\dimen0=\wd0 \advance\dimen0 by -\wd1 \divide\dimen0 by 2
	\mathord{\lower 0.1ex \hbox{\kern\dimen0\unhbox1\kern\dimen0}}
}

\begin{document}
	\title[\rmfamily Moduli Spaces of PSC-Metrics on Top. Spherical Space Forms]{\rmfamily Moduli Spaces of Metrics of Positive \\Scalar Curvature on Topological\\ Spherical Space Forms}
	% DATE
	\date{\today}
	% MATH SUBJECT CLASSIFICATION AND KEYWORDS
	\author{Philipp Reiser}
	\address{Institut f\"ur Mathematik\\Karlsruher Institut f\"ur Technologie (KIT)\\Germany.}
	\email{\href{mailto:philipp.reiser@kit.edu}{philipp.reiser@kit.edu}}
	
	\begin{abstract}
		Let $M$ be a topological spherical space form, i.e.\ a smooth manifold whose universal cover is a homotopy sphere. We determine the number of path components of the space and moduli space of Riemannian metrics with positive scalar curvature on $M$ if the dimension of $M$ is at least 5 and $M$ is not simply-connected.
	\end{abstract}

	\maketitle

	\section{Introduction and Main Result}

	Let $M$ be a closed smooth manifold. We denote by $\Rpsc(M)$ the space of Riemannian metrics with positive scalar curvature on $M$ and by $\Mpsc(M)$ the corresponding \emph{moduli space}, i.e.\ the quotient of $\Rpsc(M)$ by the pull-back action of the diffeomorphism group. We equip $\Rpsc(M)$ with the smooth topology and the moduli space with the quotient topology. These spaces, as well as spaces and moduli spaces of Riemannian metrics satisfying different curvature conditions have been studied in various ways, see e.g.\ \cite{tw15}, \cite{tu16} for an outline.
	
	This article addresses the problem of determining the number of path components of both $\Rpsc(M)$ and $\Mpsc(M)$ if $M$ is a non-simply-connected \emph{topological spherical space form}, i.e.\ a non-trivial quotient of a homotopy sphere by a free action of a finite group. In case $M$ is a \emph{linear spherical space form}, i.e.\ the action is given by an isometric action on the round sphere, then this problem has been solved if the dimension of $M$ is at\linebreak least $5$:
	
	\begin{theoremAlph}[{\cite[Theorem 0.1]{bg96}}]
		\label{thm_linSpF}
		Let $M$ be a linear spherical space form of dimension at least $5$ which is not simply-connected. Then there exists an infinite family of metrics $g_i\in\Rpsc(M)$ such that $g_i$ and $g_j$ are not concordant and lie in different path components of $\Mpsc(M)$ if $i\neq j$.
	\end{theoremAlph}
	Two metrics $g_0,g_1\in\Rpsc(M)$ are \emph{concordant} if there exists a metric\linebreak $g\in\Rpsc(M\times[0,1])$ which is product near the boundary and restricts to $g_i$\linebreak on $M\times\{i\}$. Metrics in the same path component of $\Rpsc(M)$ are concordant, see e.g.\ \cite{rs01}. Hence $\Rpsc(M)$ has an infinite number of path components. This also follows from the fact that a lower bound on the number of path components of $\Mpsc(M)$ is always a lower bound on the number of path components of $\Rpsc(M)$.
	
	In fact, Theorem \ref{thm_linSpF} is the consequence of two theorems proved in \cite{bg96} that can be applied to manifolds satisfying certain properties involving dimension, fundamental group and Spin or Pin structures, see Theorems \ref{thm_bg2} and \ref{thm_bg3} below. Theorem \ref{thm_linSpF} is then proved by verifying that all considered spherical space forms satisfy the requirements of one of these theorems. The Main Theorem of this article is the following, which is obtained by the same strategy:
	
	\begin{maintheorem}
		Let $M$ be a topological spherical space form of dimension at least $5$ which is not simply-connected and admits a metric with positive scalar curvature. Then there exists an infinite family of metrics $g_i\in\Rpsc(M)$ such that $g_i$ and $g_j$ are not concordant and lie in different path components of $\Mpsc(M)$ if $i\neq j$.
	\end{maintheorem}
	This affirmatively answers the question on page 11 of \cite{ks96}, where the authors determined under which conditions a topological spherical space form admits a metric with positive scalar curvature. Their main theorem is given as follows:
	
	\begin{theoremAlph}[{\cite[Main Theorem]{ks96}}]
		\label{thm_topSpFormPosCurv}
		Let $M$ be a topological spherical space form of dimension $n \geq5$. If $n\not\equiv1,2\mod 8$ then $M$ admits a metric with positive scalar curvature. If $n\equiv1,2\mod 8$ then $M$ admits a metric with positive scalar curvature if and only if $|\pi_1(M)|$ is even or the universal cover of $M$ admits a metric with positive scalar curvature.
	\end{theoremAlph}
	For the last case in Theorem \ref{thm_topSpFormPosCurv} note that by \cite{st92} the universal cover of the topological space form $M$ admits a metric with positive scalar curvature if and only if its alpha invariant vanishes.

	In lower dimensions the situation is completely different. In dimension 2 the only (topological) spherical space form which is not simply-connected is $\R P^2$. By \cite[Theorem 3.4]{rs01} the space $\Rpsc(\R P^2)$ is contractible, so both $\Rpsc(\R P^2)$ and $\Mpsc(\R P^2)$ are path-connected. 
	
	In dimension 3, by Perelman's proof of Thurston's Elliptization Conjecture (see \cite{mt14} or \cite[Theorem E]{dl09}), every 3-dimensional topological spherical space form $M^3$ is diffeomorphic to a linear spherical space form. By \linebreak\cite[Main Theorem]{ma12} the moduli space $\mathcal{M}^+(M^3)$ is path-connected. For recent work on the space of metrics $\Rpsc(M^3)$ we refer to \cite{bk19}. So dimension 4 is the only remaining open case.
	
	This article is organized as follows: In Section 2 we summarize some basic definitions and results on topological spherical space forms and prove some preliminary group-theoretic tools. In Section 3 we use these tools to prove the Main Theorem.
	
	\begin{ack}
		This article is based on the author’s master thesis, written at the Karlsruhe Institute of Technology (KIT) under the supervision of  Fernando Galaz-García. The author would like to thank him for his support. The author would also like to thank the Department of Mathematics of Shanghai Jiao Tong University for its hospitality during the preparation of this article.
	\end{ack}
	
	\section{Preliminaries}
	
	In this section we collect some definitions and basic results and develop the tools we will use in the proof of the Main Theorem.
	
	\begin{definition}
		A \emph{topological spherical space form} is a smooth manifold whose universal cover is a homotopy sphere. 
	\end{definition}
	
	Linear spherical space forms have been classified, see \cite{wo67}. This does not hold for the much larger class of topological spherical space forms and there are still many open problems, see e.g.\ \cite{ha14} for a survey.

	In the following $M$ denotes a topological spherical space form of dimension $n$. Its universal cover $\widetilde{M}$ is homeomorphic to $S^n$ and the fundamental group $\pi_1(M)$ is finite. Furthermore it has the following properties:
	\begin{proposition}
		\label{prop_spFormFundGr}
		If $n$ is even then $\pi_1(M)$ is trivial or $\Z_2$. If $n$ is odd then every Sylow subgroup of $\pi_1(M)$ is cyclic or a generalized quaternion group.
	\end{proposition}
	A generalized quaternion group is a group $Q_{2^{k+1}}$ of order $2^{k+1}$, $k>1$, with generators $x$ and $y$ and relations
	\begin{equation}
		\label{eq_quatRep}
		x^{2^{k-1}}=y^2,\,x^{2^{k}}=1,\,yxy^{-1}=x^{-1}.
	\end{equation}
	\begin{proof}
		If $n$ is even then, by the Lefschetz Fixed Point Theorem applied to the action on the universal cover, every non-trivial element of $\pi_1(M)$ reverses the orientation. Hence, there is at most one non-trivial element in $\pi_1(M)$.
		
		If $n$ is odd then $\pi_1(M)$ has periodic cohomology by \cite[Chapter XVI.9, Application 4]{ce56}, which is equivalent to the claim, see \cite[Theorem XII.11.6]{ce56}.
	\end{proof}

	The cyclic group $\Z_p$ acts on $S^{2m-1}\subseteq\C^m$ via
	\[(k+p\Z)\cdot(z_1,\dots,z_m)=(\lambda_1^{k}z_1,\dots,\lambda_m^{k}z_m), \]
	where $\lambda_i=\lambda^{q_i}$ for $\lambda$ a primitive $p$-th root of unity and $q_i$ and $p$ are coprime. The quotient space is called a \emph{lens space}. 
	
	\begin{proposition}
		\label{prop_cyclicSpaceForm}
		If $\pi_1(M)$ is cyclic then $M$ is homotopy equivalent to a linear spherical space form, that is, $M$ is homotopy equivalent to $S^n$, $\R P^n$ (if $n$ is even) or to a lens space (if $n$ is odd).
	\end{proposition}
	\begin{proof}
		If $n$ is even then $M$ is a homotopy sphere or $\pi_1(M)=\Z_2$ and in the latter case $M$ is homotopy equivalent to $\R P^n$ as shown in \cite[IV.3.1]{me71}. We refer to \cite[14E]{wa99} for the case where $n$ is odd.
	\end{proof}

	We now consider $\Spin$ and $\Pin^{\pm}$ structures on topological spherical space forms. Recall that the group structure of the universal cover of the orthogonal group $\textup{O}(n)$ is uniquely determined on the component which contains the identity element (this is the $\Spin$ group), but on the other component a preimage of a reflection can square to $\pm1$ and we obtain the two groups $\Pin^\pm$. The existence of $\Spin$ and $\Pin^\pm$ structures can be characterized by Stiefel-Whitney classes as follows:
	
	\begin{proposition}[{\cite[Theorem II.1.2 and Theorem II.1.7]{lm89} and \cite[Proposition 1.1.26]{ka68}}]
		Let $M$ be a smooth manifold. Denote by $w_k(M)\in H^k(M;\Z_2)$ the $k$-th Stiefel-Whitney class of its tangent bundle. Then the following hold:
		\begin{enumerate}[(a)]
			\item $M$ is orientable if and only if $w_1(M)$ vanishes.
			\item  $M$ admits a $\Spin$ structure if and only if both $w_1(M)$ and $w_2(M)$ vanish.
			\item $M$ admits a $\Pin^+$ structure if and only if $w_2(M)$ vanishes.
			\item $M$ admits a $\Pin^-$ structure if and only if $w_2(M)+w_1(M)^2$ vanishes.
		\end{enumerate}
	\end{proposition}

	Now suppose that $M$ is oriented. We equip $\widetilde{M}$ with a Riemannian metric which is invariant under the action of $G=\pi_1(M)$, so the action lifts to the oriented orthonormal frame bundle $P_\SO(\widetilde{M})$. The manifold $\widetilde{M}$ is a homotopy sphere, hence it admits a unique Spin structure $P\to P_\SO(\widetilde{M})$ and every $g\in G$ has two lifts to $P$. Denote the group of all such lifts by $\mathcal{G}$. This leads to a group extension
	\begin{equation}
		\label{eq_centralSpinExtension}
		1\longrightarrow\Z_2\longrightarrow\mathcal{G}\overset{\pi}{\longrightarrow} G\longrightarrow 1.
	\end{equation}
	Both elements of $\ker\pi$ are central in $\mathcal{G}$, so the extension is central. In order to analyze this extension we need two group-theoretic lemmas.
	
	\begin{lemma}
		\label{lemma_semiDirProd}
		Let $G$ be a finite group with a cyclic 2-Sylow subgroup $S$. Let $m\cdot 2^k=|G|$, where $m$ is odd. Then $G$ has a unique normal subgroup of order $m$.
	\end{lemma}
	\begin{proof}
		If such a subgroup exists then it is unique as it contains precisely all elements of odd order. Now the homomorphism $\ell\colon G\to\textup{Sym}(G)$ given by left-multiplication followed by $\textup{sign}\colon\textup{Sym}(G)\to\Z_2$ is surjective as every generator of $S$ has non-trivial image. Hence its kernel $H$ has index $2$ and $H\cap S$ is a cyclic 2-Sylow subgroup of order $2^{k-1}$ in $H$. Suppose $H$ has a unique normal subgroup $N$ of order $m$. Every conjugate of $N$ is contained in $H$ as $H$ is normal and, by uniqueness, it follows that it equals $N$, so $N$ is normal in $G$. Thus, the claim follows by induction.
	\end{proof}
	\begin{lemma}
		\label{l_centrZ2ext}
		Let $G$ be a finite group and let $m\cdot 2^k=|G|$, where $m$ is odd. If $G$ has a normal subgroup $N$ of order $m$ then the following hold:
		\begin{enumerate}[(a)]
			\item The group $G$ is the semi-direct product of $N$ and a 2-Sylow subgroup.
			\item The group $\mathcal{G}$ in the extension \eqref{eq_centralSpinExtension} has a normal subgroup of order $m$ that maps isomorphically to $N$.
			\item The extension \eqref{eq_centralSpinExtension} splits if and only if its restriction to a 2-Sylow subgroup splits.
		\end{enumerate}
	\end{lemma}
	\begin{proof}	
		Consider the projection $G\to G/N$. Its restriction to a 2-Sylow subgroup $S$ is an isomorphism, hence the projection splits.
		
		Now consider the extension \eqref{eq_centralSpinExtension}.	The 2-Sylow subgroups of $\pi^{-1}(N)$ have order 2, so we can apply Lemma \ref{lemma_semiDirProd} to see that $\pi^{-1}(N)$ has a normal subgroup $\mathcal{N}$ of order $m$ which maps isomorphically to $N$ under $\pi$. The group $\mathcal{N}$ is normal in $\mathcal{G}$ as it contains all elements of odd order.
		
		 It follows from $(a)$ that $\mathcal{G}$ is the semi-direct product of $\mathcal{N}$ and $\mathcal{S}=\pi^{-1}(S)$. If there is a splitting $S\to\mathcal{S}$, then we obtain a map
		\[\phi\colon G=N\rtimes S\longrightarrow\mathcal{N}\rtimes\mathcal{S}=\mathcal{G}\]
		by identifying $N$ with $\mathcal{N}$. The map $\phi$ clearly satisfies $\pi\circ\phi=\textup{id}$. Hence, it defines a splitting of \eqref{eq_centralSpinExtension} if it is a group homomorphism. The image $\phi(sns^{-1})$ equals $\phi(s)\phi(n)\phi(s)^{-1}$ for $n\in N$, $s\in S$ as both are elements of $\mathcal{N}$ and map to $sns^{-1}$. Hence $\phi$ commutes with the action of $S$ on $N$, thus it is a homomorphism on their semi-direct product $G$.
	\end{proof}
	
	The manifold $M$ admits a Spin structure if and only if the whole action of $G=\pi_1(M)$ on $P_\SO(\widetilde{M})$ can be lifted to $P$. This is the case if and only if the extension \eqref{eq_centralSpinExtension} splits, i.e. if and only if $\mathcal{G}=\Z_2\oplus G$ (as the extension is central). This leads to the following:
	
	\begin{proposition}
		\label{prop_spFormOrSpin}
		The topological spherical space form $M$ is orientable if and only if $n$ is odd or $G$ is trivial. If $n$ is odd and the 2-Sylow subgroups of $G$ are cyclic then $M$ admits a $\Spin$ structure if and only if $n\equiv 3\mod 4$ or $|G|$ is odd.
	\end{proposition}
	We will use the following lemma in the proof of the proposition.
	\begin{lemma}[{\cite[Theorem 1]{fr87}}]
		\label{l_lensSpSpin}
		If $M$ is a lens space of dimension $n=2m-1$ then $M$ is spin if and only if $m$ is even or $|\pi_1(M)|$ is odd.
	\end{lemma}
	\begin{proof}[Proof of Proposition \ref{prop_spFormOrSpin}]
		All homotopy spheres are orientable and spin, so we can assume that $G$ is non-trivial. If $n$ is even, then $M$ is homotopy equivalent to $\R P^n$ which is not orientable. Since Stiefel-Whitney classes are invariant under homotopy equivalence it follows that $M$ is not orientable. If $n$ is odd, then the action of $G$ on $\widetilde{M}$ preserves the orientation as it induces the identity on $H_n(\widetilde{M};\Q)$ by the Lefschetz Fixed Point Theorem. Hence the quotient $M$ is orientable.
		
		Now let $n=2m-1$ be odd and assume that $G$ has a cyclic 2-Sylow subgroup $S$. We consider the central extension \eqref{eq_centralSpinExtension}. By Lemma \ref{lemma_semiDirProd} there is a normal subgroup of maximal odd order and we can apply Lemma \ref{l_centrZ2ext} to see that the extension splits if and only if its restriction to $S$ splits. The quotient of the action of $S$ on $\widetilde{M}$ has cyclic fundamental group, so it is homotopy equivalent to a lens space by Proposition \ref{prop_cyclicSpaceForm}. Thus, by using the homotopy invariance of Stiefel-Whitney classes, it follows from Lemma \ref{l_lensSpSpin} that the extension splits if and only if $m$ is even or $S$ is trivial, i.e.\ $M$ admits a Spin structure if and only if $m$ is even or $|G|$ is odd.	
	\end{proof}
	\begin{remark}
		Proposition \ref{prop_spFormOrSpin} remains true if $S$ is a generalized quaternion group (cf.\ e.g.\ proof of Theorem 2.1 in \cite{ks90}). This result cannot be obtained by the same method, however, as there are central extensions of $Q_8$ by $\Z_2$ that do not split but every restriction to a cyclic subgroup splits; these are semi-direct products of $\Z_4$ and $\Z_4$. On the other hand there are no central extensions of $Q_8$ by $\Z_2$ where every element has twice the order of its image as we see in the following proposition.
	\end{remark}
	\begin{proposition}
		\label{prop_dimQuat}
		The group $Q_{2^{k+1}}$ cannot act smoothly and freely on a homotopy sphere of dimension $n\not\equiv 3\mod 4$. In particular, topological spherical space forms of dimension $n\equiv1\mod 4$ have cyclic 2-Sylow subgroups.
	\end{proposition}
	\begin{proof}
		By Proposition \ref{prop_spFormFundGr} we can assume that $n$ is odd and we again consider the extension \eqref{eq_centralSpinExtension}, where $G=Q_{2^{k+1}}$. Now let $g\in G$ be non-trivial and suppose $n\equiv 1\mod 4$. Then, by restricting the action to the cyclic subgroup generated by $g$, we obtain that both elements of $\pi^{-1}(g)$ have twice the order of $g$ as the extension does not split in this case. Denote by $\tilde{x}$ and $\tilde{y}$ preimages of the generators $x$ and $y$ of $G$, in particular $\tilde{x}$ has order $2^{k+1}$. By using the relations \eqref{eq_quatRep} we obtain:
		\[\tilde{x}^{2^{k-1}}=\pm \tilde{y}^2\text{ and }\tilde{y}\tilde{x}\tilde{y}^{-1}=\pm \tilde{x}^{-1}. \]
		Hence, 
		\[\tilde{x}^{-2^{k-1}}=\tilde{y}\tilde{x}^{2^{k-1}}\tilde{y}^{-1}=\pm \tilde{y}^{2}=\tilde{x}^{2^{k-1}},\]
		so $\tilde{x}^{2^k}$ is trivial, which is a contradiction.
	\end{proof}

	\section{Proof of the Main Theorem}
	
	Our Main Theorem will be a consequence of the following theorems:
	
	\begin{theorem}[{\cite[Theorem 0.2 and Theorem 1.1]{bg96}}]
		\label{thm_bg2}
		Let $M$ be a closed connected manifold of odd dimension $n=2m-1\geq 5$ with finite fundamental group $G$ and assume that its universal cover admits a $\Spin$ structure. Consider the central extension \eqref{eq_centralSpinExtension} given by
		\[1\longrightarrow\Z_2\longrightarrow \mathcal{G}\longrightarrow G\longrightarrow 1 \]
		and assume that the following hold:
		\begin{enumerate}[(a)]
			\item The group $\mathcal{G}$ contains an element $g\neq\pm1$ which is not conjugate to either $-g$ or to $-g^{-1}$ if $m$ is even.
			\item The group $\mathcal{G}$ contains an element $g$ which is not conjugate to either $-g$ or to $g^{-1}$ if $m$ is odd.
		\end{enumerate}
		Then, if $M$ admits a metric with positive scalar curvature, there exists an infinite family $g_i\in\Rpsc(M)$ such that $g_i$ and $g_j$ are not concordant and lie in different path components of $\Mpsc(M)$ if $i\neq j$.
	\end{theorem}
	
	This theorem is an extension of the main theorems in \cite{bg95}, where $M$ is assumed to be spin. Then $\mathcal{G}=\Z_2\oplus G$, so condition $(a)$ is satisfied if and only if $G$ is non-trivial and condition $(b)$ is satisfied if and only if $G$ has an element that is not conjugate to its inverse.
	
	\begin{theorem}[{\cite[Theorem 0.3]{bg96}}]
		\label{thm_bg3}
		Let $M$ be a closed connected manifold of even dimension $n=2m\geq6$ with fundamental group $\Z_2$ and assume that $M$ is not orientable and admits a $\Pin^\varepsilon$ structure, where $\varepsilon=\textup{sign}(-1)^m$. Then, if $M$ admits a metric with positive scalar curvature, there exists an infinite family $g_i\in\Rpsc(M)$ such that $g_i$ and $g_j$ are not concordant and lie in different path components of $\Mpsc(M)$ if $i\neq j$.
	\end{theorem}
	We now adapt, with the help of the results proven in Section 2, the proof of Theorem 0.1 in \cite{bg96} in order to prove the Main Theorem.
	\begin{proof}[Proof of the Main Theorem.]
		First assume that $n=2m-1$ is odd. Denote by $S$ a 2-Sylow subgroup of $G$. By Proposition \ref{prop_spFormFundGr} there are two possibilities: $S$ is cyclic or a generalized quaternion group. We consider each possibility separately.
		
		First assume that $S$ is cyclic. We use Proposition \ref{prop_spFormOrSpin} to determine in which cases $M$ is spin: If $m$ is even, then $M$ is spin and we can apply Theorem \ref{thm_bg2}. We can also do that if $m$ is odd and $|G|$ is odd, since then every non-trivial element of $G$ is not conjugate to its inverse. If $|G|$ is even and $m$ is odd, then $M$ is not spin. Then consider the central extension \eqref{eq_centralSpinExtension} given by
		\[1\longrightarrow \Z_2\longrightarrow\mathcal{G}\overset{\pi}{\longrightarrow} G\longrightarrow 1. \]
		It does not split since $M$ is not spin, so $\mathcal{S}=\pi^{-1}(S)$ is cyclic by Lemma \ref{l_centrZ2ext} and hence $\mathcal{G}$ is the semi-direct product of $\mathcal{S}$ and the unique maximal normal subgroup $\mathcal{N}$ of odd order by Lemmas \ref{lemma_semiDirProd} and \ref{l_centrZ2ext}. Let $g\in \mathcal{S}$. Then for $n\in \mathcal{N}$ we have
		\[n g n^{-1}=g\cdot (g^{-1} n g n^{-1}). \]
		As $g^{-1} n g n^{-1}\in \mathcal{N}$, it follows that $ngn^{-1}\in \mathcal{S}$ if and only if $g^{-1} n g n^{-1}$ is trivial and in this case $ngn^{-1}=g$. Hence, the only element of $\mathcal{S}$ to which $g$ can be conjugate is $g$ and $g$ satisfies the requirement of $(b)$ in Theorem \ref{thm_bg2} if $g\neq g^{-1}$. But $|\mathcal{S}|$ is a multiple of $4$, so such an element exists.
		
		Now consider the case where $S=Q_{2^{k+1}}$ is a generalized quaternion group. Then $m$ is even by Proposition \ref{prop_dimQuat}. If $\mathcal{G}$ has a non-trivial element $g$ of odd order, then both $-g$ and $-g^{-1}$ have even order, so $g$ is not conjugate to any of them and we can apply Theorem \ref{thm_bg2}. If there is no element of odd order, then $G=S$. Consider the generators $x$ and $y$ in the presentation \eqref{eq_quatRep}. In particular we have
		\[yxy^{-1}=x^{-1}. \]
		Denote by $\tilde{x}$ and $\tilde{y}$ preimages in $\mathcal{G}$. Then $\tilde{y}\tilde{x}\tilde{y}^{-1}=\pm\tilde{x}^{-1}$, so
		\[\tilde{y}\tilde{x}^2\tilde{y}^{-1}=\tilde{x}^{-2} \]
		holds. This shows that the conjugacy class of $\tilde{x}^{2}$ only consists of $\tilde{x}^2$ and $\tilde{x}^{-2}$. Hence $\tilde{x}^2$ is conjugate to $-\tilde{x}^2$ or $-\tilde{x}^{-2}$ if and only if $\tilde{x}^2=-\tilde{x}^{-2}$. This is the case if and only if $\tilde{x}$ has order $8$ and $x$ has order 4. By restricting the action as in the proof of Proposition \ref{prop_dimQuat} we see that the elements $x$ and $\tilde{x}$ have the same order, so $\tilde{x}^2$ is not conjugate to $-\tilde{x}^2$ or $-\tilde{x}^{-2}$ and we can apply Theorem \ref{thm_bg2}.
		
		Finally assume that $n=2m$ is even. Then $G=\Z_2$ by Proposition \ref{prop_spFormFundGr}, hence $M$ is homotopy equivalent to $\R P^n$ by Proposition \ref{prop_cyclicSpaceForm}. Denote by $a\in H^1(\R P^n;\Z_2)$ the generator of the cohomology ring. Then
		\[w_1(\R P^n)=(n+1)\cdot a=a\text{ and }w_2(\R P^n)=\frac{n(n+1)}{2}\cdot a^2. \]
		This shows that $w_2(\R P^n)=0$ if $m$ is even, so $M$ admits a $\Pin^+$ structure, and $w_2(\R P^n)+w_1(\R P^n)^2=a^2+a^2=0$ if $m$ is odd, so $M$ admits a $\Pin^-$ structure. In both cases we can apply Theorem \ref{thm_bg3}.
	\end{proof}

	\bibliographystyle{plainurl}
	\bibliography{ReferencesSpaceForms}

\end{document}